\documentclass[12pt, leqno, letterpaper]{article}

\usepackage{ifxetex}

\usepackage[letterpaper]{geometry}
\ifxetex
  \usepackage[T1]{fontenc}
  \usepackage[no-math]{fontspec}
  
  \usepackage[eucal,mtphrb,mtpfrak,noamssymbols,
  	      subscriptcorrection,zswash]{mtpro2}

  \DeclareMathSizes{12}{12}{9}{7}
\else
  \usepackage[T1]{fontenc}

  \usepackage[T1]{fontenc}
  \usepackage[variablett]{lmodern}
  
  \usepackage[LGRgreek]{mathastext}
  \MTgreekfont{lmtt} 
  \Mathastext
  \let\varepsilon\epsilon 

  \newcommand{\mbf}{\mathbf} 
  \usepackage{mathrsfs}
\fi

\usepackage[scr=boondox, scrscaled=1]{mathalfa}

\usepackage{amsthm}

\newtheoremstyle{mytheoremstyle}
{\topsep}                    
{\topsep}                    
{\slshape}                   
{0.32in}                           
{\bfseries}                  
{.}                          
{0.08in}                       
{}  

\theoremstyle{mytheoremstyle}

\makeatletter
\@ifclassloaded{report}{
   \newtheorem{theorem}{Theorem} }
{  \newtheorem{theorem}{Theorem}[section] }
\makeatother

\newtheorem{proposition}[theorem]{Proposition}
\newtheorem{lemma}[theorem]{Lemma}

\newtheorem{remark}[theorem]{Remark}

\newtheorem{assumption}[theorem]{Assumption}

\newtheorem*{theo}{Theorem}

\renewenvironment{proof}[1][\proofname]{{\scshape #1. }}
{\qed\vspace{0.16in}}

\usepackage{pgfornament}

\usepackage{titlesec}

\titleformat{\section}[runin]
{\bfseries}
{\thesection.}{0.04in}{}[.]
\titlespacing{\section}{0.32in}{0.32in}{0.08in}

\titleformat{\subsection}[runin]
{\slshape} 
{\thesubsection.}{0.04in}{}[.]
\titlespacing{\subsection}{0.32in}{0.08in}{0.08in}

\renewcommand{\abstractname}{\bf Abstract}

\renewenvironment{abstract}
{\small
\begin{center}
  \bfseries \abstractname\vspace{0in}\vspace{0pt}
\end{center}
\list{}{
	\setlength{\leftmargin}{0.5in}
    \setlength{\rightmargin}{\leftmargin}
	\itemindent \parindent} 
  \item\relax}
{\endlist}

\usepackage{caption}
\usepackage{amssymb}

\usepackage{amsmath}
\usepackage{graphicx}
\usepackage{enumitem}
\usepackage{comment}
\usepackage{url}
\usepackage{marginnote}
\usepackage{booktabs}
\usepackage{bm}
\usepackage{mdframed}
\usepackage{xfrac}

\usepackage{hyperref}
\hypersetup{colorlinks=true, allcolors=deluge}

\usepackage{harvard}
\makeatletter
\@ifpackageloaded{harvard}{}
  {\newcommand{\citeasnoun}{\cite}}
\makeatother

\newcommand{\citen}{\citeasnoun}
\newcommand{\citep}{\cite}
\newcommand{\req}[1]{(\ref{#1})}

\global\mdfdefinestyle{clean}{%
linecolor=black,linewidth=0.5pt,%
leftmargin=0.32in,rightmargin=0.32in
}

\usepackage{float}
\floatstyle{ruled}
\newfloat{algorithm}{hpt}{lop}
\floatname{algorithm}{Algorithm}

\newcommand{\s}[1][1]{\hspace{#1pt}}

\newcommand{\tq}[1]{{\textquotedblleft #1\textquotedblright}}

\newcommand{\oprst}{\textup} 

\newcommand{\indf}[1]{ 1_{ \s[-0.5] \{#1\} }}

\newcommand{\Var}{\oprst{Var}}

\newcommand{\Exp}{\oprst{E}\s[.5]}

\newcommand{\tsum}{{\textstyle\sum\nolimits}}

\newcommand{\bst}{ \s[1.5] | \s[1.5] }
\newcommand{\cst}{ \s[0.5] : \s[0.5] }

\newcommand{\tr}{\textup{tr}}

\usepackage{pgffor}

\foreach \x in {A,B,...,Z,a,b,...,z} 
{
  \expandafter\xdef\csname cal\x\endcsname{\noexpand 
	\ensuremath{\noexpand\mathcal{\x}}}
  \expandafter\xdef\csname scr\x\endcsname{\noexpand 
	\ensuremath{\noexpand\mathscr{\x}}}
  \expandafter\xdef\csname bb\x\endcsname{\noexpand 
	\ensuremath{\noexpand\mathbb{\x}}}
  \expandafter\xdef\csname rm\x\endcsname{\noexpand 
	\ensuremath{\noexpand\mathrm{\x}}}
  \expandafter\xdef\csname bf\x\endcsname{\noexpand 
	\ensuremath{\noexpand\mbf{\x}}}
}

\ifxetex
\else
\newcommand{\upalpha}{\alpha}
\newcommand{\upbeta}{\beta}
\newcommand{\upgamma}{\gamma}
\newcommand{\updelta}{\delta}
\newcommand{\upepsilon}{\epsilon}

\newcommand{\upmu}{\mu}

\newcommand{\upxi}{\xi}

\newcommand{\upphi}{\phi}

\newcommand{\uppsi}{\psi}

\fi

\newcommand{\upep}{\upepsilon}

\newcommand{\upto}{ \uparrow }

\newcommand{\limp}{\lim_{p \upto \infty}}

\newcommand{\mselet}{\textup{MSE}}
\newcommand{\mse}[2]{\mselet\s(#1\bst #2)}
\newcommand{\msp}[2]{\mselet_p (#1\bst #2)}
\newcommand{\msf}[2]{\mselet_\infty (#1\bst #2)}

\newcommand{\sphlet}{\textup{SPH}}
\newcommand{\sph}[2]{\sphlet\s(#1\bst #2)}
\newcommand{\spp}[2]{\sphlet_p (#1\bst #2)}
\newcommand{\spf}[2]{\sphlet_\infty (#1\bst #2)}

\newcommand{\ip}[2]{\langle #1 , #2 \rangle}

\newcommand{\avelet}{m}
\newcommand{\ave}[1]{\avelet(#1)}
\newcommand{\avp}[1]{\avelet_p (#1)}
\newcommand{\avf}[1]{\avelet_\infty (#1)}
\newcommand{\avsq}[1]{\avelet^2(#1)}

\newcommand{\varlet}{\rms}
\newcommand{\var}[1]{\varlet^2 (#1)}
\newcommand{\vap}[1]{\varlet^2_p (#1)}
\newcommand{\vaf}[1]{\varlet^2_\infty (#1)}

\newcommand{\sap}[1]{\varlet_p (#1)}

\newcommand{\cov}[2]{\textup{cov}( #1,#2 )}

\newcommand{\cof}[2]{\textup{cov}_\infty( #1,#2 )}

\newcommand{\dsp}{\rmr}

\newcommand{\unk}{\theta}
\newcommand{\mun}{\upmu}
\newcommand{\sig}{\sigma}
\newcommand{\xig}{\upxi}
\newcommand{\dev}{w}
\newcommand{\dvl}{\nu}
\newcommand{\del}{\updelta}

\newcommand{\dat}{\rmY}
\newcommand{\red}{\calE}
\newcommand{\res}{\upep}

\newcommand{\seig}{\scrs}
\newcommand{\csph}{\rmd}

\newcommand{\reig}{\upalpha}
\newcommand{\nvec}{\upgamma}
\newcommand{\nois}{\upphi}
\newcommand{\nprj}{\varphi}

\newcommand{\NM}{\Gamma}
\newcommand{\NS}{\rmG}

\newcommand{\xvp}{x_{p}}
\newcommand{\xvx}{x}
\newcommand{\xvf}{x_{\infty}}
\newcommand{\XX}{\rmX}

\newcommand{\Xf}{\calX}
\newcommand{\xp}{\uppsi_p}
\newcommand{\xx}{\uppsi}
\newcommand{\xf}{\uppsi_\infty}
\newcommand{\snr}{\textup{SNR}}
\newcommand{\dst}{\chi}

\ifxetex
  \newcommand{\hJS}{h^{\s\textup{\bf JS}}}
  \newcommand{\etaJS}{\eta^{\s\textup{\bf JS}}}
\else
  \newcommand{\hJS}{h^{\s\textup{JS}}}
  \newcommand{\etaJS}{\eta^{\s\textup{JS}}}
\fi

\usepackage{color}

\definecolor{red}{rgb}{1,0,0}
\definecolor{gray}{rgb}{0.5,0.5,0.5}
\definecolor{darkgray}{rgb}{0.4,0.4,0.4}
\definecolor{blue}{rgb}{0,0,1}
\definecolor{green}{rgb}{0,1,0}

\definecolor{deluge}{RGB}{124, 113, 173}
\definecolor{bamboo}{RGB}{220, 92, 5}
\definecolor{yellow}{RGB}{255, 172, 0}
\definecolor{orange}{RGB}{255, 144, 0}
\definecolor{oyster}{RGB}{151, 139, 125}
\definecolor{coral}{RGB}{199, 186, 167}
\definecolor{downy}{RGB}{110, 197, 184}

\title{{\Large \bf James-Stein estimation
of the first principal component}}

\author{
Alex Shkolnik\footnote{Department of Statistics 
and Applied Probability, University of California, 
Santa Barbara, CA and Consortium for Data Analytics in Risk, 
University of California, Berkeley, CA.
Email:~{\tt shkolnik@ucsb.edu}.}
}

\date{August 31, 2021 \\ \vspace{0.08in}
This Version: 
\today.\footnote{I am indebted to Lisa Goldberg for 
conjecturing the result derived in this work.}}

\begin{document}

\ifxetex
  \let\lsum\sum
  \renewcommand{\sum}{\bm{\lsum}}

  \let\lprod\prod
  \renewcommand{\prod}{\bm{\lprod}}
\else
\fi


\maketitle

\begin{abstract} 
The Stein paradox has played an influential role in the field of
high dimensional statistics.  This result warns that the sample
mean, classically regarded as the \tq{usual estimator}, may be
suboptimal in high dimensions.  The development of the
James-Stein estimator, that addresses this paradox, has by now
inspired a large literature on the theme of \tq{shrinkage} in
statistics. In this direction, we develop a James-Stein type
estimator for the first principal component of a high dimension
and low sample size data set. This estimator shrinks the
usual estimator, an eigenvector of a sample covariance matrix 
under a spiked covariance model, and yields superior asymptotic
guarantees. Our derivation draws a close
connection to the original James-Stein formula
so that the motivation and recipe for
 shrinkage is intuited in a natural way.
\end{abstract}

\vspace{0.16in}
\newpage

\section{Introduction} \label{sec:intro}
The Stein paradox has played an influential role in the field of
high dimensional statistics. This result warns that the sample
mean, classically regarded as the \tq{usual estimator}, may be
suboptimal in high dimensions. In particular, \citen{stein1956}
showed that the usual estimator of a location parameter $\unk
\in \bbR^p$ from uncorrelated Gaussian observations becomes
inadmissible when $p > 2$ under a mean-squared error criterion.
That is, an estimator with a uniformly lower risk must exist.
That estimator was established by \citen{james1961} and
eponymously named.

Among the numerous perspectives that motivate the James-Stein
estimator,\footnote{A few examples include the Galtonian
regression perspective promoted by \citen{stigler1990}, the
purely frequentist development of the estimator in
\citen{gupta1991} and the geometrical explanation in
\citen{brown2012} that builds on Stein's original heuristic
argument \citep[Section 1]{stein1956}.} the empirical Bayes
perspective (see \citen{efron1975}) is particularly elegant.
Letting $\eta$ denote the sample mean computed with $n$ measurements
of an unknown $\theta \in \bbR$ and assuming an additive,
normally distributed error $\dev$ that has a zero mean and
a variance $\dvl^2$ (e.g., $\dvl = \del /\sqrt{n}$ where each
measurement has standard error $\del$), we write
\begin{align} \label{jsprob}
  \eta = \unk + \dev \s .
\end{align}
Taking a Gaussian prior on the unknown 
$\unk$, that is independent of $\dev$, implies that
\begin{align} \label{ce}
 \Exp (\unk \bst \eta) = \Exp (\eta)
+ \bigg( 1 - \frac{\dvl^2}{\Var(\eta)} \bigg)
 (\eta - \Exp (\eta)) \s ,
\end{align}
the bivariate-normal conditional expectation formula.  While, by
definition of conditional expectation, $\Exp(\unk \bst \eta)$ is
the best estimator of $\unk$ in the sense of mean-squared error,
it cannot be implemented directly as the first two moments of
$\eta$ are unknown.\footnote{Often, $\dvl$ is assumed to be
known but estimates $\hat{\dvl}$ can also be used as done in
\citen{james1961}.} Stein's paradox now amounts to the fact that
\tq{good} substitutes for $\Exp(\eta)$ and $\Var(\eta)$ are
available only in higher dimensions; precisely, when
$\unk \in \bbR^p$ with $p > 2$.

Formula $\req{ce}$ extends easily to the
multivariate\footnote{See formulas in 
\citen[Section 2.5]{anderson2003} which 
can be used to design James-Stein estimators of an unkown
$\unk \in \bbR^p$ when $\dev$ has a general  covariance 
$\Sigma$ as in \citen{bock1975}.} case and 
motivates the estimator
\begin{align} \label{jshrink}
  \eta(\rmc) = \avelet
 + \rmc 
  \s (\eta - \avelet)
\end{align}
where $\avelet$ is an estimate (or guess) of the expected
value
of $\eta \in \bbR^p$ and $\rmc \in (0,1)$ is a shrinkage 
parameter. In words, $\req{jshrink}$ attempts to 
center the entries of $\eta$, 
shrinks the resulting entries and recenters at $\avelet$. 
Assuming $\dvl$ is known and $p > 2$, setting
\begin{align} \label{jsc}
 \rmc = 1 - 
 \s \frac{\dvl^2}{\var{\eta}}
 \Big( \frac{p-2}{p} \Big)
\end{align}
where $\var{\eta} = \sum_{i=1}^p (\eta_i - \avelet_i)^2 / p$
yields the James-Stein estimator.  Remarkably, any fixed
$\avelet \in \bbR^p$ (e.g., \citen{stein1956} considers the 
origin)\footnote{
A natural choice for $\avelet$ takes the sample mean 
of $\eta$ in each entry, referred to as shrinkage toward the 
\tq{grand mean}. However, this would 
require $p > 3$ \citep{efron1975}).
}
results in an estimator $\req{jshrink}$ with a strictly smaller
mean-squared error than $\eta$  \citep[Section 1]{efron1975}.
While three is provably the critical dimension\footnote{
The theme of a critical dimension is 
encountered frequently in statistics and probability.
\citen{brown1971}, for example, derives a close mathematical
relationship between the admissibility of the James-Stein
estimator and the transience of the Brownian 
motion in $\bbR^p$, which also requires $p > 2$.}, 
\citen[Section 1]{stein1956} heuristically argued 
for higher performance in higher dimensions and
for relaxing the normality.

 
The James-Stein estimator has inspired a
rich literature on the theme of \tq{shrinkage} in statistics.
Just a small sampling of examples includes ridge regression
\citep{hoerl1970}, the LASSO \citep{tibshirani1996}, the
Ledoit-Wolf covariance estimator \citep{ledoit2004} and the
Elastic Net \citep{zou2005}. Excellent textbook treatments of
the ideas behind the Stein paradox and James-Stein shrinkage
include \citen{gruber2017} and \citen{fourdrinier2018}. In this
paper, we leverage these ideas to develop and analyze a
James-Stein estimator for the first principal component of a
sample covariance matrix. The results again prove the
efficacy of James-Stein estimation, and do so for one the
cornerstone methods in high-dimensional statistics,
principal component analysis \cite{jolliffe2016}.

Consider a $p \times p$ sample covariance matrix $\rmS$
that is based on $n$ observations of some random vector
$y \in \bbR^p$. Without loss of generality, we write
\begin{align}
  \rmS 
 = \seig^2_p h h^\top  + \NS
\end{align}
for $\NS = \rmS - \seig^2_p\s h h^\top$ and $h$, 
the sample eigenvector with the largest eigenvalue, i.e., 
\begin{align} \label{hands}
   \rmS\s  h = \seig^2_p \s h \s[16] \text{and} \s[16]
  \seig^2_p = \max_{|z|=1} \ip{z}{\rmS z}   \s .
\end{align}
By convention $h$ has unit length and corresponds to a direction
along which the variance of $\rmS$ (i.e., $\seig^2_p$) is
maximum (i.e., the first principal component). A substantial and
rapidly growing literature exists to study the ($p$ and/or $n$)
asymptotic behavior of the eigenpair $(\seig^2_p, h)$ and the
remaining eigenstructure in order to quantify either the
estimation or the empirical error.  See \citen{wang2017} for
recent results and a systematic discussion of this literature.
The topic of shrinkage estimators arises naturally in this
context and was raised by \citen{stein1986}, who suggested
improving the usual estimate $\rmS$ via eigenvalue shrinkage.
Indeed, there is by now a large literature on estimators that
adjust the eigenvalues of sample covariance matrices to improve
their performance with respect to some loss function
\citep{donoho2018}.

In this paper, we develop and analyze a James-Stein estimator
for the first principal component of a high-dimension and
low-sample (HDLS) data set.\footnote{The HDLS framework,
as introduced in \citen{hall2005} and \citen{ahn2007},
is increasingly relevant for data
science  \cite{aoshima2018}.} 
The recipe for the estimator
begins with $h$ and $\seig^2_p$ in $\req{hands}$ and the next 
$(q-1)$ largest sample eigenvalues $\seig^2_{p-1},
\dots, \seig^2_{p-q+1}$ ($\min(n,p)>q$) corresponding
to a model with $q$ spikes.\footnote{Roughly speaking
$q$ is the number of factors (or spikes) in the data, after
which a sufficiently large eigengap (between the
$q$th and the next eigenvalue) is observed
(see \citen{fan2020}).}

\begin{enumerate}[leftmargin=0.75in, rightmargin=0.75in]
\item[\bf Step 1.] Set $\eta = \seig_p h$,
compute the sample statistics
$\ave{\eta} = \sum_{i=1}^p \eta_i/p$ and
$\var{\eta} = \sum_{i=1}^p (\eta_i-\ave{\eta})^2/p$,
and define
\begin{align*} 
 \rmc = 1 - \frac{\hat{\dvl}^2}{\var{\eta}} 
\hspace{0.16in} \text{where} \hspace{0.16in}
\hat{\dvl}^2   =
\bigg( \frac{\tr(\rmS)-
(\seig^2_p+\cdots +\seig^2_{p-q+1})}{\min(n,p)-q} \bigg) 
 \s[2]/ p \s .
\end{align*}
\item[\bf Step 2.] Return the estimator
(corrected principal component)\footnote{With a slight
abuse of the notation, $u - x = (u_1-x,\dots, u_p-x)$
for $u \in \bbR^p$ and $x \in \bbR$.\label{ft:abuse}}
\begin{align*} 
 \hJS = \frac{1}{\sqrt{p}} \s \bigg(
 \frac{\ave{\eta} + \rmc\s[1.5] (\eta - \ave{\eta})}
 {\sqrt{\avsq{\eta} + \rmc^2\s \var{\eta}}} 
 \bigg) \s .
\end{align*}
\end{enumerate}

The vector $\hJS$ is the James-Stein estimator of the 
first principal component of the data. The numerator contains the shrinkage
formula $\req{jshrink}$ while the divisor normalizes the
shrunk vector to a unit length (by convention). The
relationship to the shrinkage parameter $\rmc$ in $\req{jsc}$
is evident by treating $p$ as large. The estimate $\hat{\dvl}^2$
corresponds to the bulk of the eigenvalue spectrum, and may
be viewed the \tq{noise} in the context of a signal-to-noise
ratio that plays a prominent role of the results
in Sections \ref{sec:model} \& \ref{sec:js}.

It is reasonable to suspect that a James-Stein type shrinkage of
the principal component $h$, a high dimensional vector, could
improve either the convergence rate or accuracy of the limit in
some appropriate asymptotic regime. However, the standard
orthonormal transformation and the eigengap partition of the
sample eigenvectors, that is typically leveraged by their
asymptotic analyses (e.g., \citen{paul2007}, \citen{shen2016}
and \citen{wang2017}), can obscure the systematic nature of the
sample bias. As sensibly pointed out by \citen{wang2017} in
reference to the partition of the sample eigenvector, the
\tq{\it two parts intertwine in such a way that correction for
the biases of estimating eigenvectors is almost impossible.}
However, in the original (untransformed) coordinate system and
the HDLS asymptotic regime, the bias can in fact be identified,
characterized and (partially) corrected. This program was
carried out in \citen{goldberg2021}, who adopt a factor model in
an HDLS regime and utilize a portfolio theory application to
motivate their analysis. 

The main results of this paper rederive the adjustment of
\citen{goldberg2021} but within a James-Stein type framework.
In particular, we establish identity $\req{jsprob}$ in which
$\eta = \seig_p h$ and $\unk$ is related to the associated
population eigenvector. From here, the James-Stein shrinkage
acts on the perturbation $\dev$ so that the estimator $\hJS$
outperforms $h$ on the mean-squared error and angle metrics as
$p \upto \infty$.  The theoretical guarantees provided here are
new and their proofs rely on a different set of mathematical
tools than \citen{goldberg2021}. In particular, the new approach
leverages Weyl's inequality and the Davis-Kahan theorem from
matrix perturbation theory to give simpler proofs and
potentially expand the scope of applicability of the resulting
estimator.  The HDLS regime, in which the number of
variables~$p$ grows to infinity, and the number of
observations~$n$ to be fixed, plays a crucial role in the
analysis. 


The paper is organized as follows. Section~2 defines the spiked
covariance model underlying our results.  Section~3 develops the
James-Stein estimator $\hJS$ and Section~4 proves the
theoretical guarantees for this estimator.  Appendix A contains
proofs of the auxiliary results.  The following notation is used
throughout. Let $\ip{u}{v}$ denote the standard inner product of
$u,v \in \bbR^d$ so that $|u| = \sqrt{\ip{u}{u}}$ and $\ave{u} =
\ip{u}{\rme}/d$ where $\rme = (1, \dots, 1)^\top$ are the length
and mean.  Set $\var{u} = |u-\ave{u}|^2/d$ and $\cov{u}{v} =
\ip{u-\ave{u}}{v-\ave{v}}/d$ (see the notation of footnote
\ref{ft:abuse}). We use a subscript $1 \le p \le \infty$ to
highlight the dependence on $p$ of various quantities, e.g.,
$\ave{\eta} = \avp{\eta}$ for $\eta \in \bbR^p$ and $\avf{\eta}$
is the limit $\limp \avp{\eta}$ when it exists.

\section{A scarcely sampled spiked model} \label{sec:model}
We use a spiked covariance model, borrowed from the
HDLS literature. 
We also restrict ourselves to a single unbounded
spike in the \tq{boundary case} (see \citen{jung2012})
wherein the largest eigenvalue of the covariance matrix
grows linearly in $p$.
In particular, consider a mean-zero (w.l.o.g.) random vector
$y \in \bbR^p$ with a $p \times p$ 
covariance matrix $\Sigma = \Var(y)$ and let, 
\begin{align} \label{popcov}
  \Sigma = \NM +  \beta \beta^\top
\end{align}
for a symmetric, positive-semidefinite
$p \times p$ matrix $\NM$ and vector $\beta \in \bbR^p$.
The following affirms that $b=\beta/|\beta|$
is an eigenvector of $\Sigma$ with eigenvalue
$\ip{\beta}{\beta}$.

\begin{assumption}[w.l.o.g.] \label{asm:mub}
$\NM b = 0$ and $\avp{b}\ge 0$ for any $p$.
\end{assumption}

To state our additional assumptions on the model,
we project the data 
vector $y \in \bbR^p$ onto the eigenvectors of $\Sigma$. 
More precisely, define
\begin{align} \label{xp}
   \xp = \ip{\beta}{y} / \ip{\beta}{\beta} \s .
\end{align}
It is immediate that $\Exp(\xp) = 0$
and $\Var(\xp) = 1$. For every eigenvalue 
$\reig_i$ of $\NM$, let 
\begin{align} \label{epsi}
   \nois_i = \ip{\nvec}{y} \s[16] \text{where} \s[16]
   \nvec \in \bbR^p 
  \s[2] \cst \s[2] \NM \s \nvec = \reig_i \s \nvec\s .
\end{align}

Note, $\Exp (\nois_i) = 0$ and $\Var(\nois_i) = \reig_i$.  As
the dimension $p$ grows we obtain a sequence $\{\nois_i\}_{i \ge
1}$. As a technical remark, $\nois_i = \nois_{i,p}$ and $\reig_i
= \reig_{i,p}$ depend on $p$. We consider a sequence of models
$\req{popcov}$ constructed from sequences $\{\beta_i\}_{i\ge 1}$
and $\{\NM_p\}$.

\begin{assumption} \label{asm:matrix}
For constants $\mun\in \bbR$ and
$\sig,\del \in (0,\infty)$ as $p \upto \infty$ we have:
\begin{enumerate}
\item[\textup{(i)}] 
$\avf{\beta}=\mun$ and $\vaf{\beta} = \sig^2$.
\item[\textup{(ii)}] $\xf = \limp \xp$ exists as a $\bbR$-valued
random variable almost surely.
\item[\textup{(iii)}] 
$\avf{\nois}=\avf{\nprj} = 0$ almost surely
for $\{\nprj_i\}_{i \ge 1}$
with $\nprj_i= \nprj_{i,p}
= \nois_i p \ave{\nvec^i}$.
\item[\textup{(iv)}] 
$\vaf{\nois} = \avf{\reig} = \del^2$ almost surely.
\end{enumerate}
\end{assumption}

Condition {(i)} imposes regularity on the sequence
$\{\beta_i\}_{i \ge 1}$ and implies that the largest eigenvalue
of $\Sigma$ (i.e., $\ip{\beta}{\beta}$) grows linearly with $p$.
The random variable $\xf$ in {(ii)} is closely related to a
principal component score (in the limit $p \upto \infty$), and
it captures the randomness along the first (population)
principal component.  Conditions {(iii)} and {(iv)} are related
to certain requirements on a measure of sphericity of the model
(e.g., $\tr(\Sigma)^2/(\tr(\Sigma^2)p)$) and summarize the
conclusions of \citen[Theorem 1]{jung2009}. In particular,
{(iii)} may be viewed as laws of large numbers for
$\{\nois_i\}_{i \ge 1}$ and $\{\nprj_i\}_{i \ge 1}$ and suggests
$p \ave{\nvec^i} \asymp 1$ whereas {(i)} implies the first
principal component has $\sqrt{p}\ave{b} \asymp 1$.  This
highlights that the spike eigenvector $b$ differs from those of
$\NM$ both in terms of the magnitude of its eigenvalue as well
as the structure of the vector itself.  According to {(iv)}, the
average eigenvalue of $\NM$ (unlike the spike eigenvalue
$\ip{\beta}{\beta}$) is bounded in $p$, and per {(iii)}, the
eigenvectors $\nvec^i$ do not have entries biased towards a
nonzero mean $\sqrt{p}\ave{\nvec^i}$, unlike $b$ (an exception
is $\mun = 0$, which is the case when the James-Stein estimator
will turn out to be ineffective; see Remark \ref{rem:muzero}).


The following assumption is a standard one in statistical
data analysis but may be (and often is) relaxed in the HDLS
setup (e.g., \citen{jung2012}).

\begin{assumption} \label{asm:iid}
There are a fixed $n \ge 2$ 
i.i.d. observations of $y \in \bbR^p$.
\end{assumption}

Our forthcoming results hold even when only $2$ observations 
are available, hence, a scarcely sampled model.
Let $\dat$ be the $p \times n$ data matrix with
the $k$th column containing the $k$th observation of $y$, and 
define the sample covariance matrix $\rmS$ by
\begin{align}
  \rmS = \dat \dat^\top /n \s .
\end{align}
We let $h \in \bbR^p$ denote the eigenvector of $\rmS$ with the
largest eigenvalue $\seig^2_p$ (see $\req{hands}$). It is 
unique only up to sign (and $|h|=1$), motivating
the following (c.f. Assumption \ref{asm:mub}).

\begin{assumption}[w.l.o.g.]\label{asm:muh}
  $\avp{h} \ge 0$ for any $p$.
\end{assumption}


We write $\rmS = \NS + \eta \eta^\top$ 
in analogy to $\req{popcov}$ (taking 
$\NS = \rmS - \eta\eta^\top$) and set
\begin{align} \label{etah}
  \eta = \seig_p h \s .
\end{align}
Next, define the following measure of finite-sample distortion,
\begin{align} \label{dst}
   \dst_n^2 = \ip{\Xf}{\Xf} / \s[-0.25] n
   \s[16] \text{and} \s[16]
    \Xf = \limp \frac{\dat^\top \beta}{\ip{\beta}{\beta}} \s .
\end{align}
The latter limit exists by Assumption \ref{asm:matrix}
while Assumption \ref{asm:iid} implies that 
$\Xf \in \bbR^n$ has i.i.d. entries (distributed
as $\xf$).
Consequently, we have $\dst^2_n \to 1$ as $n \upto \infty$.

We can measure the error in any estimator $\eta$ 
of $\unk$ by the mean-squared error, as would be 
consistent with the James-Stein framework.
\begin{align} \label{mse}
 \msp{\eta}{\unk} = \ip{\eta - \unk}{\eta-\unk} /p
\end{align}

\begin{proposition} \label{prop:mse} 
Let $\unk = \dst_n \beta$ and suppose
Assumptions \ref{asm:mub}--\ref{asm:muh} hold. Then,
\begin{align}
  \msf{\eta}{\unk} = \frac{\s \del^2}{n \s} \s .
\end{align}
\end{proposition}
\begin{remark}  \label{rem:dst}
If $\unk = \beta$ then the right side
would be multiplied by the factor,
$1 +  \frac{\snr^2}{\dsp^2_\infty}\s  
\Big(\frac{\dst_n-1}{\dst_n}\Big)^2$
where we define a signal-to-noise ratio $\snr$ and 
signal-incoherence $\dsp_\infty$ as
\begin{align} \label{snr}
  \snr = \Big( \frac{\sig}{\del} \Big) \s \dst_n \sqrt{n} 
  \hspace{0.16in} \text{and} \hspace{0.16in}
  \dsp_\infty = \frac{1}{\sqrt{1 + (\mun/\sig)^2}} \s .
\end{align}
For $\snr$, we regard $\sig \dst_n$
as a distorted signal, and $\del/\sqrt{n}$ as noise that
vanishes as $n$ grows (also, $\sig \dst_n \to \sig$). 
The signal-incoherence $\dsp_\infty$ is the limit of
$\dsp_p = \dsp_p(\beta) = \sap{\beta}/|\beta|$ 
(per Assumption \ref{asm:matrix}) determined
by the signal-to-noise ratio $\mun/\sig$ of the 
vector $\beta$.  A large value of $\dsp_\infty$ corresponds to more
variation, or \tq{incoherence} in the entries 
$\{\beta_i\}_{i\ge 1}$
\end{remark}

A more standard way to evaluate the goodness of a
sample eigenvector $h$ is via its angle away from its
population counterpart $b = \beta/|\beta|$. To this end, let
\begin{align} \label{sph}
 \spp{h}{b} = \spp{\eta}{\unk} 
= \sin^2 \Big( \s[-2]
\arccos \frac{\ip{\eta}{\unk}} {|\eta| |\unk|} \Big) \s .
\end{align}

\begin{proposition}  \label{prop:sph}
Let $\unk = \dst_n \beta$ and suppose
Assumptions \ref{asm:mub}--\ref{asm:muh} hold. Then,
\begin{align}  \label{spherr}
 \spf{\eta}{\unk} = \frac{\dsp^2_\infty}{\dsp^2_\infty + \snr^2}
 \s[2] .
\end{align}
\end{proposition}
\begin{remark}  \label{rem:muzero}
Clearly, $\req{spherr}$ also holds with  $\unk = \beta$.
\end{remark}

\section{James-Stein estimation of sample eigenvectors}\label{sec:js}
Having established that the sample eigenvector corresponding
to the spike (i.e., the largest eigenvalue) carries
finite sample error, it is natural to ask whether James-Stein
shrinkage can improve this \tq{usual} estimator.
The key to this question is the (to be established) identity
\begin{align} \label{jsform}
  \eta = \unk + \dev
  \s[16] \text{and} \s[16]
  \unk =  \dst_n \beta
\end{align}
for a random vector $\dev \in \bbR^p$ specified in 
$\req{dev}$ of Section \ref{sec:key}. As in definition
$\req{etah}$, 
we have $\eta = 
\seig_p h$ where $h$ is the
sample eigenvector with the largest eigenvalue, $\seig^2_p$.
The perturbation $\dev$ of $\unk$ turns out to be such that
the shrinkage of $\eta$ is effective.
We remark that the recipe of Section \ref{sec:intro} extends 
our derivation of the James-Stein estimator below to the 
case of multiple (there, $q$) spikes in a natural way. This
extension is effective because 
the eigenvectors corresponding to the spikes are 
mutually orthogonal,
but it is suboptimal. An optimal estimator in the 
multi-spiked setup is left for future work.

\subsection{The JS estimator}
Equation $\req{jsform}$ establishes a relationship between 
the sample and population eigenvectors that suggests
a James-Stein estimator may be derived. An informal
derivation proceeds as follows. Consider the 
shrinkage parameter
\begin{align} \label{jsceig}
 \rmc = 1 - \frac{\hat{\dvl}^2}{\var{\eta}} 
\end{align}
based on $\req{jsc}$ with $\hat{\dvl}$, an estimate
of the \tq{noise}. It is reasonable to assign the latter to 
be the average of the non-spiked, non-zero eigenvalues of  
$\rmS$. That is,
\begin{align} \label{noisest}
  \hspace{1.28in}
  \hat{\dvl}^2  
= 
\bigg( \frac{\tr(\rmS)-\seig^2_p}{n-1} \bigg) \s[2]/ p
  \hspace{0.64in} (p \ge n),
\end{align}
where the scaling by $p$ turns out to be necessary due to
the counterintuitive behavior of the HDLS asymptotics.
When $p < n$ the divisor $n-1$ must be replaced by $p-1$.

This paves the way for the James-Stein sample
eigenvector estimate,
\begin{align*}
\etaJS 
= \ave{\eta} + \rmc \s[2] (\eta - \ave{\eta})
\end{align*}
of the unnormalized eigenvector and by convention, we 
take unit length version,
\begin{align} \label{hJS}
 \hJS = \frac{\etaJS}{|\etaJS|}  
 = \frac{1}{\sqrt{p}} \s \bigg(
 \frac{\ave{h} + \rmc\s[1.5] (h - \ave{h})}
 {\sqrt{\avsq{h} + \rmc^2\s \var{h}}} 
 \bigg)
\end{align}
as the James-Stein estimator of the population
eigenvector $b = \beta/|\beta|$.

The following James-Stein type theorems characterize the
improvement due to shrinkage in the original mean-squared sense
as well as in the angle metrics.

\begin{theorem}  \label{thm:js_mse}
Suppose  Assumptions \ref{asm:mub}--\ref{asm:muh} hold. 
Then, almost surely,
\begin{align*}
  \msf{\etaJS}{\unk} &= 
  \rmc_\infty \s \msf{\eta}{\unk}  
\end{align*}
where $\rmc_\infty \in (0,1)$ is the limit of $\rmc_p = \rmc$
in $\req{jsceig}$ 
with $\snr$ defined in $\req{snr}$ and
\begin{align} \label{clim}
  \rmc_\infty = \frac{\snr^2}{1 + \snr^2}  \s .
\end{align}
\end{theorem}

\begin{theorem}  \label{thm:js_sph}
Suppose  Assumptions \ref{asm:mub}--\ref{asm:muh} hold. 
Then, almost surely,
\begin{align*}
 \spf{\etaJS}{\unk} 
 =\spf{\hJS}{b} 
 &= \csph_\infty 
  \s \spf{h}{b}
\end{align*}
where $\csph_\infty \in [\rmc_\infty,1]$ 
where $\rmc_\infty$ is in $\req{clim}$ and with
$\snr$ and $\dsp_\infty$ in $\req{snr}$, we have
\begin{align} \label{slim}
 \csph_\infty = \rmc_\infty + \frac{\dsp^2_\infty}
   {1+\snr^2} \s .
\end{align}
\end{theorem}

Related results may be found in \citen{goldberg2020}
and \citen{goldberg2021} but the
metrics there are motivated by solutions of certain
quadratic programs that are useful 
in finance and portfolio theory.

\begin{remark} \label{rem:dev}
Note, $\csph_\infty=1$ (i.e., no improvement in angle)
if and only if $\mun=0$.
\end{remark}

\subsection{The geometry of Stein's paradox} 
We shed insight into the James-Stein estimator in $\req{hJS}$ 
by deriving general conditions under which Theorems
\ref{thm:js_mse} and \ref{thm:js_sph} hold.
Our analysis adopts the pure frequentist perspective 
in \citen{gupta1991} and supplements it by illustrating 
the Euclidean and the spherical geometry of the estimator.
The two geometries reflect the definitions of the error
($\mselet$ and $\sphlet$)
in the two theorems.

\newcommand{\jsarg}{c}
\newcommand{\unkv}{\upxi}
\newcommand{\unkm}{\avelet}

\begin{lemma}  \label{lem:cmin}
Let $\eta = \unk + \dev$ for $\unk,\dev \in \bbR^p$. Then,
the solutions of the optimizations 
$\min_{\jsarg \in \bbR} \mse{\eta(\jsarg)}{\unk}$ 
and 
$\min_{\jsarg \in \bbR} \sph{\eta(\jsarg)}{\unk}$ 
(see $\req{mse}$ and $\req{sph}$)
are given by
\begin{align} \label{cex}
  \rmc^{\mselet} = 
  \frac{\cov{\unk}{\eta}}{\var{\eta}}  
  \s[16] \text{and} \s[16]
  \rmc^{\sphlet} = \frac{\ave{\eta}}{\ave{\unk}} \s[2]
  \rmc^{\mselet} \s .
\end{align}
\end{lemma}

The next assumptions may be viewed as laws of large
numbers in the random setting or regularity
conditions in a deterministic one. They concern
the sequences $\{\unk_i\}_{i=1}^\infty$ and 
$\{\dev_i\}_{i=1}^\infty$, and allow for dependence
on $p$ (i.e., 
$\unk_i=\unk_i^{(p)}$ and
$\dev_i=\dev_i^{(p)}$).

\begin{assumption}  For constants $\avelet \in \bbR$ and 
$\dvl,\xig \in (0,\infty)$ as $p \upto \infty$, we have:
\begin{enumerate}
\item[\textup{(i)}] 
$\avf{\unk} = \unkm$ and $\vaf{\unk} = \unkv^2$,
\item[\textup{(ii)}] $\avf{\dev} = 0$ and $\vaf{\dev} = \dvl^2$,
\item[\textup{(iii)}] $\cof{\unk}{\dev} = 0\s$,
\item[\textup{(iv)}] there exists an estimator 
$\hat{\dvl} = \hat{\dvl}_p$ for each $p$
with $\hat{\dvl}_\infty =\dvl$.
\end{enumerate}
\label{asm:jsgeom}
\end{assumption}

The following identities follow by direct calculation.

\begin{lemma}  \label{lem:basics}
Suppose $\{\unk_i\}$ and $\{\dev_i\}$ satisfy
Assumption \ref{asm:jsgeom} and $\eta_i = \unk_i + \dev_i$.
Then (almost surely), $\avf{\eta} = \unkm$,
$\cof{\eta}{\unk} = \unkv^2$ 
and $\vaf{\eta} = \unkv^2 + \dvl^2$.
\end{lemma}

We define the signal-to-noise ratio $\snr$ 
and the signal-incoherence $\dsp_\infty$ as
\begin{align} \label{asnr}
\snr = \frac{\unkv}{\dvl} 
\hspace{0.16in} \text{and} \hspace{0.16in}
\dsp_\infty = 
\frac{1}{\sqrt{1 + (\unkm\s[-1]/\unkv)^2}} \s ,
\end{align}
which are compatible with $\req{snr}$ 
upon taking
$\unk =\dst_n \beta$ and $\dvl = \del/\sqrt{n}$.
The following result establishes the conclusions
of Theorems \ref{thm:js_mse} and \ref{thm:js_sph}
in our abstract setting.

\begin{proposition} \label{prop:js}
Let $\eta = \unk + \dev$ where $\unk,\dev \in \bbR^p$ and an
estimator $\hat{\dvl}$ satisfy 
Assumption~\ref{asm:jsgeom}. Then, for
$\rmc_\infty$ and $\csph_\infty$ defined in $\req{clim}$ and 
$\req{slim}$ but with $\snr$ and $\dsp_\infty$ in $\req{asnr}$
the estimate
$\eta(\rmc) = \eta + \rmc \s (\eta - \ave{\eta})$
with parameter
$\rmc = 1 - \frac{\hat{\dvl}^2}{\var{\eta}}$ satisfies
\begin{align*}
  \msf{\eta(\rmc)}{\unk}
  = \rmc_\infty \s \msf{\eta}{\unk} 
  \hspace{0.08in} \text{and} \hspace{0.08in}
  \spf{\eta(\rmc)}{\unk}
  = \csph_\infty \s \spf{\eta}{\unk} \s .
\end{align*}
Moreover, the optimal parameters
$\rmc^{\mselet}$ and
$\rmc^{\sphlet}$ in $\req{cex}$ converge as $p \upto \infty$ 
to $\rmc_\infty$.
\end{proposition}

Figure $\ref{fig:js}$ illustrates the geometry of the 
estimator $\eta(\rmc)$ of the vector $\unk$.

\begin{figure}[htp!]
\begin{center}
\includegraphics[width=0.875\textwidth]{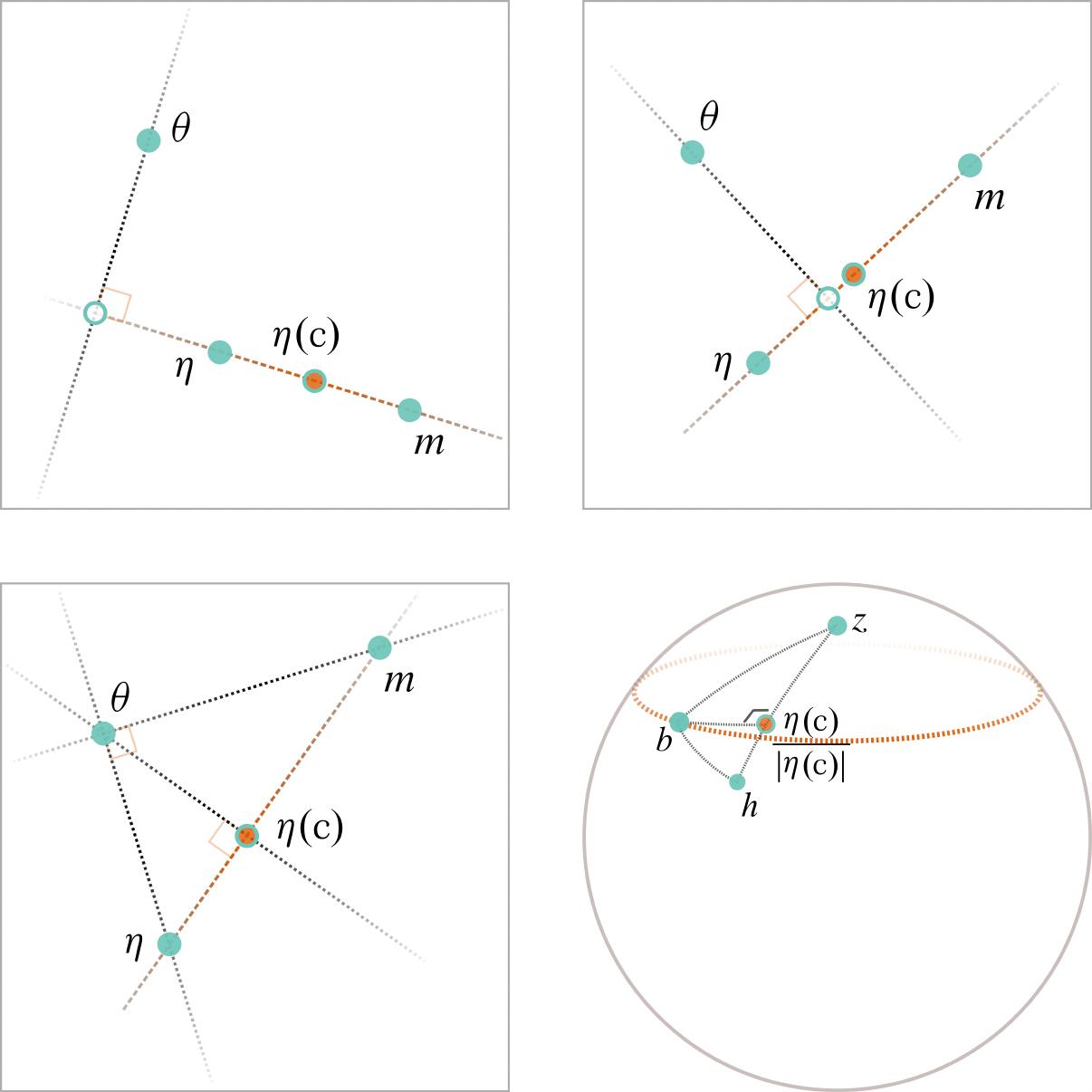}
\end{center}
\caption{Illustration of the estimator $\eta(\rmc)$ in low
\emph{(top left)},
high \emph{(top right)}, 
and limiting \emph{(bottom left)} dimensions, 
relative to the shrinkage
target $\avelet \in \bbR^p$, the
vector with all entries equal to $\ave{\eta}$. The
open circle marks the estimator with the optimal
shrinkage parameter $\rmc^{\mselet}$. In Euclidean
geometry, the estimator
$\eta(\rmc)$ is located anywhere on the ray originating
at the observation $\eta$ and passing through $\avelet$,
because $\rmc \le 1$ while $\eta(1) = \eta$
and $\eta(0) = \avelet$. The spherical geometry
\emph{(bottom right)} presents the limiting ($p = \infty$)
analog of the illustration in the bottom left. There, 
$b = \unk/|\unk|$, $h = \eta/|\eta|$ and $z = \avelet
/|\avelet|$ and the contour describes all vectors with 
mean entry $\ave{b}$, which highlights the difference between the
two geometries.
 }
\label{fig:js_sphere}
\label{fig:js}
\end{figure}

\section{Proofs the main results} \label{sec:proofs}
We proceed by the following three main steps.
\begin{enumerate}
\setlength\itemsep{0.02in}
\item[(1)] We establish the key identity $\eta
= \unk + \dev$ with $\unk = \dst_n \beta$ per $\req{jsform}$
in Section~\ref{sec:key}.
 \item[(2)] We derive the convergence properties 
in the HDLS regime of the eigenvalues
and eigenvectors of $\rmS$ under our spiked covariance model
setting in Section \ref{sec:dualconv}.
\item[(3)] We verify that $\unk,\dev$ in (1) 
and 
$\hat{\dvl}$ in $\req{noisest}$
satisfy the 
conditions of
Assumption \ref{asm:jsgeom}, which leads to the
guarantees for James-Stein shrinkage in Proposition
\ref{prop:js}.
\end{enumerate}

Theorems \ref{thm:js_mse} and \ref{thm:js_sph} and then
corollaries of Proposition \ref{prop:js}, 
which is proved in Appendix \ref{app:aux}.
We will make use of two classic results in matrix perturbation 
theory. 

\begin{theo}[Weyl] Let $\rmA$ and $(\rmA + \Delta)$ be 
(real) symmetric $n \times n$ matrices with eigenvalues 
$\upalpha_1 \ge  \cdots \ge \upalpha_n$ and 
$\zeta_1 \ge \cdots \ge \zeta_n$ respectively. Then,
 \begin{align*}
  \max_{1 \le j \le n} |\upalpha_j - \zeta_j| 
  \le |\Delta |  \s .
\end{align*}
\end{theo}

For a proof see \citen{horn2012} (also \citen{weyl1912}).

\begin{theo}[Davis-Kahan] Let $\rmA$ and $(\rmA + \Delta)$ be 
(real) symmetric $n \times n$ matrices with
$\rmA\s a^j = \upalpha_j a^j$ and 
$(\rmA + \Delta)\s b^j = \upbeta_j b^j$ 
for eigenvectors $a^j,b^j \in \bbR^n$ and 
eigenvalues $\upalpha_j,\upbeta_j \in \bbR$. Suppose
$\upalpha_1 \ge \cdots \ge \upalpha_n$
and $\upbeta_1 \ge \cdots \ge \upbeta_n$ with
the convention $\upalpha_0 = \infty
= -\upalpha_{n+1}$ and assume 
$\upgamma_j = \min\{
\upalpha_{j-1}- \upalpha_j, \upalpha_j - \upalpha_{j+1}\} > 0$. 
Then,
\begin{align*}
  \hspace{0.32in}
  |\s a^j - b^j | \le \frac{3}{\s[2]\upgamma_j} \s |\Delta|
  \hspace{0.32in}
  \text{provided (w.l.o.g.) } \ip{a^j}{b^j} \ge 0 .
\end{align*}
\end{theo}

This result is proved in \citen[Corollary 1]{yu2015}.

\subsection{Establishing the key identity} \label{sec:key}
A key tool for random matrix theory in the HDLS regime is
the dual sample covariance matrix. This is $n \times n$ matrix
($n \ge 2$ is fixed),
\begin{align}  \label{dual}
  \rmL = \dat^\top \dat / p \s .
\end{align}

The next result is well known and relates 
the spectra of $\rmS = \dat \dat^\top/n$ and $\rmL$.

\begin{lemma} \label{lem:key}
Let $\rmL u = \ell^2 u$ where $\ell^2 \in (0,\infty)$ 
and $u \in \bbR^n$.
Then, $\rmS v = \seig^2 v$ where
$v = \dat u / (\sqrt{p} \ell)$ 
and $\seig^2 = \ell^2 p /n$. Conversely, 
let $\rmS v = \seig^2 v$ where
$\seig^2 \in (0,\infty)$ 
and $v \in \bbR^p$. Then, $\rmL u = \ell^2 u$ where
$u = \dat^\top v /(\sqrt{n}\seig)$ and $\ell^2 = \seig^2 n/p$.
\end{lemma}

\begin{proof} 
Multiplying the identity 
$\rmL u = \ell^2 u$ by $\dat$ from both sides, we obtain
\begin{align*}
   \dat \rmL u = \ell^2 \dat u 
  \hspace{0.16in} \Rightarrow \hspace{0.16in}
   \rmS \dat u = \Big( \frac{\ell^2 p}{n} \Big) \dat u \s .
\end{align*}
Note, $(\dat u)^\top (\dat u)= (u^\top \rmL u) p = \ell^2 p$,
so $v = (\dat u) / (\sqrt{p} \ell)$ has unit length. 
Dividing by $\sqrt{p} \ell$
yields $\rmS v = \seig^2 v$ as required.
The converse has an identical argument.
\end{proof}

The spike model $\Sigma = \NM + \beta \beta^\top$ has a full 
basis of eigenvectors given by $b = \beta/|\beta|$ and 
$\{\nvec^i\}_{i=1}^{p-1}$, the latter corresponding to the
nonzero eigenvalues $\{\reig\}_{i=1}^{p-1}$ of $\NM$. Thus,
\begin{align*}
  y &= \ip{b}{y} b + 
  \tsum_{i=1}^p \ip{\nvec^i}{y}\s \nvec^i
  = \beta \xx  + \res
\end{align*}
where $\res = \tsum_{i=1}^p \nois_i \nvec^i \in \bbR^p$
and $\xx = \xp = \ip{\beta}{y}/\ip{\beta}{\beta}$ as in
$\req{xp}$. Consequently, letting $\dat$ denote the $p \times
n$ matrix of i.i.d. observations of $y \in \bbR^p$, we have
\begin{align*}
\dat = \beta \XX^\top + \red \s .
\end{align*}
where $\XX =\dat^\top \beta /\ip{\beta}{\beta}
 \in \bbR^n$ consists of i.i.d. observations of $\xx$
and $\red$ is a $p \times n$ matrix with i.i.d. columns
consisting of the observations of $\res$ as defined above.

By orthogonality, we obtain that
$\rmL = \dat^\top \dat/p= (\ip{\beta}{\beta}/p) \s[2] 
  \XX \XX^\top + \red\red^\top /p$. Let $\xvp$ be the
eigenvector of $\rmL$ with the largest eigenvalue, $\ell_p^2$,
and $\xvf= \Xf/\dst_n \in \bbR^n$ 
(the unit length normalization of $\Xf$), 
where $\Xf$ and $\dst_n$ are 
defined in $\req{dst}$.
By Lemma \ref{lem:key},
\begin{align*}
  h &= \frac{\rmY \xvp}{\sqrt{p}\ell_p}
  = \frac{(\beta \XX^\top + \red) \xvx}{\sqrt{p} \ell_p} 
  = \beta \s \frac{\dst_n \ip{\xvp}{\xvf}}{\sqrt{p/n} \ell_p}
  + \frac{\red \xvp}{\sqrt{p} \ell_p} \\
  &= \beta \Big( \frac{\dst_n}{\seig_p} \Big)
  +  \beta \Big( \frac{\dst_n}{\seig_p}\Big) 
  \ip{\xvp}{\xvp-\xvf}
  + \frac{\red \xvp}{\sqrt{n} \seig_p} 
\end{align*}
We deduce that $\eta = \seig_p h = \dst_n \beta + \dev$ 
as required by $\req{jsform}$ where
\begin{align} \label{dev}
  \dev = \dst_n \beta  \s
  \ip{\xvp}{\xvp-\xvf}
  + \frac{\red \xvp}{\sqrt{n}}  \s .
\end{align}

\subsection{Convergence of the eigen- values/vectors} 
\label{sec:dualconv} It
is not difficult to establish that limit  of $\rmL
= \rmL^{(p)}$ as $p \upto \infty$ 
(in any norm on $\bbR^{n\times n}$)
takes the following form
\begin{align*}
  \rmL^{(\infty)} = (\sig^2 + \mun^2) \s (n \dst^2_n)
   \s[2] \xvf \xvf^\top + \del^2 \rmI \s .
\end{align*}
The first term is the limit of 
$(\ip{\beta}{\beta}/p) \s \XX \XX^\top$ under Assumption
\ref{asm:matrix} (note that
$\ip{\beta}{\beta} = \var{\beta} + \avsq{\beta}$)
and the definitions of $\XX$ and $\xvf$ above. By
Assumption \ref{asm:iid}, the columns
of $\red$ are i.i.d. copies of $\res$ with 
$\Exp (\res) = 0_p$ by definition, the strong law of
large numbers, confirms the that the off-diagonal entries
of the second term are zero. That $\del^2$ determines all
the diagonal entries is again a consequence of Assumption 
\ref{asm:matrix}.
This entails proving that $\ip{\res}{\res}/p$ 
converges almost surely to $\del^2$ as is done
in Section \ref{sec:verify} item (ii).

The matrix $\rmL^{(\infty)}$ has an easily described spectrum.
Its largest eigenvalue is given by 
$\ell_\infty^2 = (\sig^2 + \mun^2) \s (n \dst^2_n) + \del^2$ 
and has the 
eigenvector $\xvf$. All remaining eigenvalues equal $\del^2$.
Since $\rmL^{(p)}$ converges (in any norm) to $\rmL^{(\infty)}$,
the largest eigenvalue $\ell^2_p$ converges to $\ell^2_\infty$
almost surely. All the remaining eigenvalues converge
to $\del^2$.

By Weyl's inequality,
setting $\rmA = \Omega$ and  $\rmB = (\rmL - \Omega)$ 
so that $\rmA + \rmB = \rmL$, we have
\begin{align}
  |\ell^2_{p-i+1} - 
((\sig^2 + \mun^2) \s (n \dst^2_n)\indf{i=1} + \del^2)| 
   &\le |\rmL^{(p)} - \rmL^{(\infty)}|
 \s[16] 1 \le i \le n\s .
\end{align}

We immediately deduce (by Lemma \ref{lem:key}) the following
result, variants of which appear in \citen{jung2012},
\citen{shen2016} and \citen{goldberg2021}. Let $\seig^2_{p-i+1}$
denote the $i$th largest eigenvalue of $\rmS$ (for $i >
\min(p,n)$ all eigenvalues are zero).

\begin{proposition} \label{prop:weyl}
 Fix $n\ge 2$ and suppose Assumptions
\ref{asm:mub}, \ref{asm:matrix} and \ref{asm:iid} hold. 
Then,
$\limp \seig^2_{p-i+1}/p = \indf{i=1} \dst_n^2 
(\sig^2 + \mun^2)  + \del^2/n$ almost surely for 
fixed $1 \le i \le n$.
\end{proposition}

Next, by the Davis-Kahan
theorem with $\rmL^{(p)} = \rmL^{\infty} + \Delta$
and $\Delta = \rmL^{(p)} - \rmL^{(\infty)}$, 
for $\xvp$ and $\xvf$ the eigenvectors of 
$\rmL^{(p)}$ and $\rmL^{(\infty)}$ with largest
eigenvalues respectively,
\begin{align}
  |\xvp-\xvf| \le \frac{3\s[2]}{\del^2}
 \s[2]  |\rmL^{(p)} - \rmL^{(\infty)}| \s.
\end{align}
Note the condition $\ip{\xvp}{\xvf} \ge 0$ is without loss
of generality as the orientation of the eigenvectors
is always arbitrary. The following result follows
immediately.

\begin{proposition} \label{prop:dk}
 Fix $n\ge 2$ and suppose that Assumptions
\ref{asm:mub}, \ref{asm:matrix} and \ref{asm:iid} hold. 
Then, we have 
$|\xvp - \xvf| \to 0$ as $p \upto\infty$ almost surely.
\end{proposition}

\subsection{Verifying Assumption \ref{asm:jsgeom}}
\label{sec:verify}
We make use of Assumptions \ref{asm:mub}--\ref{asm:matrix}.
There are four items to verify for $\unk = \dst_n\beta$
and $\dev$ in $\req{dev}$.
Take, $\avelet = \dst_n \mun$, $\unkv = \dst_n \sig$
and $\dvl = \del/\sqrt{n}$.
\begin{enumerate}
\item[(i)] We have $\avf{\unk} = \dst_n \mun$ and
$\vaf{\unk} = \dst_n^2 \s \sig^2$ by 
Assumption \ref{asm:matrix} part (i).
\item[(ii)] To see that $\avf{\dev} = 0$, we compute 
$\ave{\dev}$ using 
$\req{dev}$ which gives
\[ \ave{\dev} = \dst_n \ave{\beta} \ip{\xvp}{\xvp-\xvf}
+ \ave{\red \xvp}/\sqrt{n}
\]
The first term vanishes by Proposition
\ref{prop:dk} since $
|\ip{\xvp}{\xvp-\xvf}| \le |\xvf-\xvp|$ and
$\avf{\beta}$ and $\dst_n$ are finite
almost surely by Assumption \ref{asm:matrix} part (ii). The
second term vanishes because
$\ave{\red \xvp}$ may be written as a linear combination of 
a fixed $n$ realizations 
of $\ave{\res}$ with coefficients being the entries of $\xvp$,
and $|\xvf|$ is bounded.
To this end, $\ave{\res} = \avp{\res}  = \tsum_{i=1}^p
\nois_i \ave{\nvec^i} = \avp{\nprj}$ 
which tends to zero as $p \upto \infty$
(almost surely) by Assumption \ref{asm:matrix} part (iii).

Similarly, to verify that $\vaf{\dev} = \dvl$, we use 
$\req{dev}$ again to calculate that
\[ \var{\dev} 
= \dst_n^2 (\ip{\beta}{\beta}/p) \ip{\xvp}{\xvp-\xvf}^2
+ \ip{\red\xvp}{\red\xvp}/(pn) - \avsq{\dev}
\] 
where we have used the fact that $\ip{\beta}{\red} = 0_n$
by Assumption \ref{asm:mub}. Since $\dst_n$ and the limit
$\dst_n^2 (\sig^2 + \mun^2)$ of $\ip{\beta}{\beta}/p$, as
above, by Proposition \ref{prop:dk} the first term 
tends to zero  as $p \upto \infty$. For the second
term, we note that $\ip{\red \xvp}{\red \xvp}$ may
be written as convex combination of a fixed $n$ 
realizations of $\ip{\res}{\res}$ with coefficients
being the entries of $\xvp$ squared, and 
$|\xvp|^2 = |\xvf|^2=1$. We have $\ip{\res}{\res}/p
= \sum_{i=1}^p \nois_i^2/p = \var{\nois} + \avsq{\nois}$ 
as the $\{\nvec^i\}_{i=1}^p$ are orthonormal.
By Assumption \ref{asm:matrix} parts (iii) and (iv),
we have $\vaf{\nois} = \del^2$ and $\avf{\nois} = 0$.
It follows that the second term converges to 
$\del^2 |\xvf|^2/n = \del^2/n = \dvl^2$.
The last term
tends to zero since $\avf{\dev}=0$ as above, and the
claim now follows.
\item[(iii)] We have $\cov{\unk}{\dev} = \ip{\unk}{\dev} /p
 - \ave{\unk}\ave{\dev}$ and since $\avf{\unk}$ is finite, the
second term vanishes as $\avf{\dev}$ as above. Again by
$\req{dev}$ and since $\ip{\beta}{\red}=0_n$,
\begin{align*}
 \ip{\unk}{\dev}/p 
  = \dst_n^2 (\ip{\beta}{\beta} /p)\s \ip{\xvp}{\xvp-\xvf}
\end{align*}
which vanishes in the limit 
by the same arguments as in (ii) above.
\item[(iv)] To see that $\hat{\dvl} = \hat{\dvl}_p$ 
in $\req{noisest}$
is an asymptotically exact estimate of 
$\dvl = \del/\sqrt{n}$, we use Proposition \ref{prop:weyl}.
Since $(\tr(\rmS)-\seig^2_p)/p
= \sum_{i=2}^n \seig^2_{p-i+1} /p$, under the hypotheses
of Proposition \ref{prop:weyl} converges almost surely
to $\del^2 (n-1) / n$,  
$\hat{\dvl}_\infty = \dvl$.
\end{enumerate}


\appendix

\section{Auxiliary proofs} \label{app:aux}
As shown in Section \ref{sec:verify}, 
the hypotheses of Proposition \ref{prop:mse}
and Proposition \ref{prop:sph} (i.e., Assumptions~{\ref{asm:mub}
--\ref{asm:muh}}) guarantee that the conditions on  
$\{\unk_i\}_{i \ge 1}$ and $\{\dev_i\}_{i\ge 1}$ 
in Assumption \ref{asm:jsgeom} are satisfied.
Consequently the proofs of these two results,
as well as that of Proposition \ref{prop:js} which
requires Assumption \ref{asm:jsgeom} directly,
reduce to the calculations below. 
The proof of Lemma
\ref{lem:basics} is omitted as it is elementary and
that of Lemma \ref{lem:cmin} is a direct
consequence of some of the expressions below.

For any $\rmc \in \bbR$ and $\eta(\rmc)
= \ave{\eta} + \rmc \s (\eta -\ave{\eta})$, 
by direct calculation
\begin{align*} 
\mse{\eta(\rmc)}{\unk}
&= (\ave{\eta}-\ave{\unk})^2  + \var{\unk} 
  + \rmc^2 \var{\eta} - 2 \rmc \s \cov{\eta}{\unk}  \s .
\end{align*}
When $\rmc = 1$ for which $\eta(1)=\eta$ and applying
Lemma \ref{lem:basics} yields
\begin{align*}
\msf{\eta}{\unk}
&= \unkv^2 + (\unkv^2
 + \dvl^2) - 2 \unkv^2 = \dvl^2 = \del^2/n
\end{align*}
which proves Proposition \ref{prop:mse}. 
The sine of the angle squared metric is computed as 
\begin{align*} 
\sph{\eta(\rmc)}{\unk}  = 1 - \bigg(
\frac{\ip{\eta(\rmc)}{\unk}}{|\eta(\rmc)| |\unk|} \bigg)^2 
= 1 - \frac{\big(\ave{\eta}\ave{\unk} 
+ \rmc \s \cov{\eta}{\unk}\big)^2}
{(\avsq{\eta}+\rmc^2 \var{\eta})(\avsq{\unk}+\var{\unk})} \s.
\end{align*}
Using the raw estimate $\eta = \eta(1)$ for which
$\rmc  =1$ we deduce by Lemma \ref{lem:basics} that
\begin{align*}
\spf{\eta}{\unk} 
= 1- \frac{\avelet^2 + \unkv^2}
{\avelet^2+\unkv^2 + \dvl^2}
= \frac{\dvl^2}{\unkv^2 + \avelet^2 + \dvl^2}
= \frac{\dsp^2_\infty}{\snr^2 + \dsp^2_\infty}
\end{align*}
which establishes Proposition \ref{prop:sph}
with $\snr$ and $\dsp_\infty$ in $\req{snr}$
(c.f. $\req{asnr}$).

Note that minimizing the expressions for
$\mse{\eta(\rmc)}{\unk}$ and $\sph{\eta(\rmc)}{\unk}$ above
over $\rmc \in \bbR$ yields the $\rmc^{\mselet}$ and 
$\rmc^{\sphlet}$ in $\req{cex}$ proving Lemma \ref{lem:cmin}.
The limits as $p \upto \infty$ of these quantities is 
easily verified as $\rmc_\infty
= \frac{\snr^2}{1 + \snr^2}$ in $\req{clim}$ 
using Lemma \ref{lem:basics}. This establishes the last
part of Proposition \ref{prop:js}. To prove the first
part, we again apply Lemma \ref{lem:basics} to deduce that
$\rmc = \rmc_p = 1 - \frac{\hat{\dvl}_p}{\vap{\eta}}
\sim 1 - \frac{\dvl^2}{\unkv^2 + \dvl^2}
= \rmc_\infty$ as above. Also,
\begin{align*}
\msf{\eta(\rmc)}{\unk}
&= \unkv^2 + \rmc^2_\infty (\unkv^2 + \dvl^2) 
 - 2\rmc_\infty \unkv^2 \\
&= \unkv^2 + \rmc_\infty \unkv^2  - 2\rmc_\infty \unkv^2  \\
&= \unkv^2 (1 - \rmc_\infty) \\
&= \frac{ \unkv^2}{1 + \snr^2} = \rmc_\infty \dvl^2
= \rmc_\infty \msf{\eta}{\unk}
\end{align*}
as required. Similarly, by Lemma \ref{lem:basics}, we 
derive that
\begin{align*}
\spf{\eta(\rmc)}{\unk} 
&= 1- \frac{(\avelet^2 + \unkv^2 \rmc_\infty)^2}
{(\avelet^2 + \rmc_\infty^2 (\unkv^2 + \dvl^2))
(\mun^2 + \sig^2)}
= (1-\rmc_\infty)\s \dsp^2_\infty  \\
&= (1-\rmc_\infty) (\snr^2 + \dsp^2_\infty) \spf{\eta}{\unk}
\\&= \bigg(\frac{\snr^2 + \dsp^2_\infty}{1 + \snr^2} \bigg)
 \spf{\eta}{\unk}
= \csph_\infty \spf{\eta}{\unk}
\end{align*}
for $\csph_\infty$ as in $\req{slim}$. This concludes
the proof of Proposition \ref{prop:js}.

\bibliographystyle{agsm}
\bibliography{jspc}

\end{document}